\documentclass[english]{article}
\usepackage[T1]{fontenc}
\usepackage[latin9]{inputenc}
\usepackage{amsmath}
\usepackage{amsthm}
\usepackage{amssymb}

\makeatletter

\usepackage{amsthm}

\theoremstyle{plain}
\theoremstyle{remark}

\usepackage{fullpage}\usepackage{makeidx}\usepackage{amsfonts}\usepackage{latexsym}

\def\C{{\mathbb{C}}}

\def\Res{{\rm Res}}
\theoremstyle{definition}
\newtheorem{lemma}{Lemma}[section]
\newtheorem{theorem}[lemma]{Theorem}\newtheorem{corollary}[lemma]{Corollary}\newtheorem{proposition}[lemma]{Proposition}\newtheorem{definition}[lemma]{Definition}\newtheorem{remark}[lemma]{Remark}\usepackage{times}

\usepackage[blocks]{authblk}
\title{Permutation orbifolds and associative algebras}

\author{Chongying Dong\footnote{supported by the China NSF grant 11871351  and the Simons foundation  634104 }}
\affil{Department of Mathematics, University of
California, Santa Cruz, CA 95064 USA}
\author{Feng Xu}
\affil{Department of Mathematics, University of California at Riverside, Riverside, CA 92521 USA}
\author{Nina Yu\footnote{supported by the China NSF grant 11971396 }}
\affil{School of Mathematical Sciences, Xiamen University, Xiamen, Fujian 361005, CHINA}

\usepackage{babel}

\makeatother

\usepackage{babel}
\begin{document}
\maketitle
\begin{abstract}
Let $V$ be a vertex operator algebra and $g=\left(1\ 2\ \cdots k\right)$
be a $k$-cycle which is viewed as an automorphism of the vertex operator
algebra $V^{\otimes k}$. It is proved that Dong-Li-Mason's associated
associative algebra $A_{g}\left(V^{\otimes k}\right)$ is isomorphic
to Zhu's algebra $A\left(V\right)$ explicitly. This result recovers
a previous result that there is a one-to-one correspondence between
irreducible $g$-twisted $V^{\otimes k}$-modules and irreducible
$V$-modules.
\end{abstract}

\section{Introduction}

Let $V$ be a vertex operator algebra and $G$ be a finite automorphism
group of $V$. Then the space of $G$-invariants $V^{G}$ itself is
also a vertex operator algebra. The so-called \emph{orbifold theory}
studies $V^{G}$ and its representation theory. The key ingredients
in the study of orbifold theory are the $g$-twisted $V$-modules
(for $g\in G$) which are not $V$-modules unless $g=1$, but restricts
to $V^{G}$-modules. From \cite{DLM5} we know that there is a connection
between $V$-modules and $g$-twisted $V$-modules for any automorphism
$g$ of finite order. However, how to construct twisted modules in
general is still a challenging problem. The permutation orbifolds
study the representations of the tensor product vertex operator algebra
$V^{\otimes k}$ with the natural action of the symmetric group $S_{k}$
as an automorphism group, where $k$ is a positive integer. A systematic
study of permutation orbifolds in the context of vertex operator algebras
was started in \cite{BDM}, where a connection between twisted modules
for tensor product vertex operator algebra $V^{\otimes k}$ with respect
to permutation automorphisms and $V$-modules was found. Specifically,
let $\sigma$ be a $k$-cycle which is viewed as an automorphism of
$V^{\otimes k}$. Then, for any $V$-module $(W,Y_{W}(\cdot,z))$,
a canonical $\sigma$-twisted $V^{\otimes k}$-module structure on
$W$ was obtained using $\Delta_{k}$-operator.

On the other hand, given a vertex operator algebra $V$ and an automorphism
$g$ of finite order $T$, an associative algebra $A_{g}\left(V\right)$
was constructed in \cite{DLM3} with the property that there is a
bijective correspondence between the sets of equivalence classes of
simple $A_{g}\left(V\right)$-modules and the sets of equivalence
classes of simple admissible $g$-twisted $V$-modules. In the case
$g=1$, $A_{g}(V)$ is exactly the associative algebra $A(V)$ \cite{Z}.
In this paper, we consider the vertex operator $V^{\otimes k}$ and
$g=\left(1\ 2\cdots k\right)\in S_{k}$ which is an automorphism of
$V^{\otimes k}.$ We study the associative algebra $A_{g}\left(V^{\otimes k}\right)$.
From \cite{BDM} we know that there should be an algebra isomorphism
between $A_{g}\left(V^{\otimes k}\right)$ and $A\left(V\right).$
But it is nontrivial to determine this isomorphism. Our main result
gives an explicit isomorphism from $A_{g}\left(V^{\otimes k}\right)$
to $A\left(V\right).$ Consequently, we give a new proof of that there
is a one-to-one correspondence between irreducible $g$-twisted $V^{\otimes k}$-modules
and irreducible $V$-modules \cite{BDM}. This may pave the way for computing the associtive algebras for fixed point algebras.

The paper is organized as follows: In Section 2, we recall some basic
notions and properties in vertex operator algebra theory, we also
recall the algebra $A_{g}\left(V\right)$ from \cite{DLM3}. In Section
3, we first study some properties of $\Delta_{k}$-operator. Then
we use $\Delta_{k}$-operator to construct an explicit isomorphism
between $A_{g}\left(V^{\otimes k}\right)$ and $A$$\left(V\right).$

\section{Basics}

In this section we recall some basic notions and fix some necessary
notations.

Let $\left(V,Y,\mathbf{1},\omega\right)$ be a vertex operator algebra
(cf. \cite{FLM2}, \cite{FHL}, \cite{Bo}). First we recall properties
in formal calculus \cite{FLM2}. We denote the space of formal power
series by
\[
V\left[\left[x\right]\right]=\left\{ \sum_{n\in\mathbb{N}}v_{n}x^{n}|v_{n}\in V\right\}
\]
and the space of truncated formal Laurent series is denoted by
\[
V\left(\left(x\right)\right)=\left\{ \sum_{n\in\mathbb{Z}}v_{n}x^{n}|v_{n}\in V,v_{n}=0\ \text{for }n\ \text{sufficiently\ negative}\right\} .
\]
We will need to use the following formula for change of variables:
For $g\left(z\right)=\sum_{m\ge M}v_{m}z^{m}\in V\left(\left(z\right)\right)$
and $f\left(x\right)=\sum_{n=1}^{\infty}a_{n}x^{n}\in\mathbb{C}\left[\left[z\right]\right]$
with $a_{1}\not=0$, the power series $g\left(f\left(x\right)\right)\in V\left(\left(x\right)\right)$
is defined by
\[
g\left(f\left(x\right)\right)=\sum_{m\ge M}v_{m}f\left(x\right)^{m}=\sum_{m\ge M}\sum_{j=0}^{\infty}v_{m}\left(a_{1}x\right)^{m}\left(_{j}^{m}\right)\left(\sum_{i=2}^{\infty}\frac{a_{i}}{a_{1}}x^{i-1}\right)^{j}.
\]
We have the following formula for the change of variables:

\begin{equation}
\text{Res}_{z}g\left(z\right)=\text{Res}_{x}\left(g\left(f\left(x\right)\right)\right)\frac{d}{dx}f\left(x\right).\label{change of variable formula}
\end{equation}

Recall that an \emph{automorphism} of a vertex operator algebra $V$
is a linear isomorphism $g$ of $V$ such that $g\left(\omega\right)=\omega$
and $gY\left(v,z\right)g^{-1}=Y\left(gv,z\right)$ for any $v\in V$.
Denote the group of all automorphisms of $V$ by $\mbox{Aut}\left(V\right).$

Let $g$ be a finite order automorphism of $V$ with order $T$. Then
\begin{align}
V=\oplus_{r=0}^{T-1}V^{r},
\end{align}
where $V^{r}=\left\{ v\in V\ |\ gv=e^{2\pi ir/T}v\right\} $ for $r\in\mathbb{Z}$.
Note that for $r,s\in\mathbb{Z}$, $V^{r}=V^{s}$ if $r\equiv s\ ({\rm mod}\;k)$.

Now we review the definition of an admissible $g$-twisted $V$-module
for a finite order automorphism $g$ of $V$ (see \cite{FLM2}, \cite{DLM3}).

\begin{definition}A \emph{weak $g$-twisted $V$-module} is a $\mathbb{C}$-linear
vector space $M$ with a linear map $Y_{M}(\cdot,z):\ V\to\left(\text{End}M\right)[[z^{1/T},z^{-1/T}]]$
given by $v\mapsto Y_{M}\left(v,z\right)=\sum_{n\in\frac{1}{T}\mathbb{Z}}v_{n}z^{-n-1}$
such that for all $v\in V$, $w\in M$ the following hold:

(1) $v_{n}w=0\ $ for $n$ sufficiently large;

(2) $Y_{M}(u,z)=\sum_{n\in\frac{r}{T}+\mathbb{Z}}u_{n}z^{-n-1}$ for
$u\in V^{r}$ with $0\le r\le T-1$;

(3) $Y_{M}(\mathbf{1},z)=Id_{M}$;

(4) For $u\in V^{r},$
\begin{align*}
 & z_{0}^{-1}\text{\ensuremath{\delta}}\left(\frac{z_{1}-z_{2}}{z_{0}}\right)Y_{M}(u,z_{1})Y_{M}(v,z_{2})-z_{0}^{-1}\delta\left(\frac{z_{2}-z_{1}}{-z_{0}}\right)Y_{M}(v,z_{2})Y_{M}(u,z_{1})\\
 & =z_{2}^{-1}\left(\frac{z_{1}-z_{0}}{z_{2}}\right)^{-\frac{r}{T}}\delta\left(\frac{z_{1}-z_{0}}{z_{2}}\right)Y_{M}\left(Y\left(u,z_{0}\right)v,z_{2}\right)
\end{align*}

where $\delta\left(z\right)=\sum_{n\in\mathbb{Z}}z^{n}$.

\begin{definition}

A $g$-\emph{twisted $V$-module} is a weak $g$-twisted $V$-module
$M$ which carries a $\mathbb{C}$-grading induced by the spectrum
of $L(0)$ where $L(0)$ is the component operator of $Y(\omega,z)=\sum_{n\in\mathbb{Z}}L(n)z^{-n-2}.$
That is, we have
\[
M=\bigoplus_{\lambda\in\mathbb{C}}M_{\lambda},
\]
where $M_{\lambda}=\left\{ w\in M\ |\ L(0)w=\lambda w\right\} $.
Moreover, it is required that $\dim M_{\lambda}$ is finite for all
$\lambda$ and for fixed $\lambda_{0},$ $M_{\frac{n}{T}+\lambda_{0}}=0$
for all small enough integers $n.$ \end{definition}

In this situation, if $w\in M_{\lambda}$ we refer to $\lambda$ as
the \emph{weight }of $w$ and write $\lambda=\text{wt}w$. If $g=1$
then this defines a $V$-modules.

\begin{definition}An \emph{admissible $g$-twisted $V$-module} is
a weak $g$-twisted module with a $\frac{1}{T}\mathbb{Z}_{+}$-grading
$M=\oplus_{n\in\frac{1}{T}\mathbb{Z}_{+}}M(n)$ such that $u_{m}M\left(n\right)\subset M\left(\mbox{wt}u-m-1+n\right)$
for homogeneous $u\in V$ and $m,n\in\frac{1}{T}\mathbb{Z}.$ $ $
\end{definition}

Note that if $M=\oplus_{n\in\frac{1}{T}\mathbb{Z}_{+}}M(n)$ is an
irreducible admissible $g$-twisted $V$-module, then there is a complex
number $\lambda_{M}$ such that $L(0)|_{M(n)}=\lambda_{M}+n$ for
all $n.$ As a convention, we assume $M(0)\ne0$, and $\lambda_{M}$
is called the \emph{weight} or\emph{ conformal weight }of $M.$

\end{definition}

\subsection{\label{subsec:The-associative-algebra}The associative algebra $A_{g}\left(V\right)$}

Let $V,g$ and $V^{r}$ be as defined above. An associative algebra
$A_{g}\left(V\right)$ was constructed in \cite{DLM3}. For homogeneous
$u\in V^{r}$ and $v\in V$, define

\[
u\circ_{g}v=\begin{cases}
\text{Res}_{z}\frac{\left(1+z\right)^{\text{wt}u}}{z^{2}}Y\left(u,z\right)v & \text{if\ }r=0\\
\text{Res}_{z}\frac{\left(1+z\right)^{\text{wt}u-1+\frac{r}{T}}}{z}Y\left(u,z\right)v & \text{if\ }r\not=0
\end{cases}
\]
and
\[
u\ast_{g}v=\begin{cases}
\text{Res}_{z}Y\left(u,z\right)v\frac{\left(1+z\right)^{\text{wt}u}}{z} & \text{if}\ r=0\\
0 & \text{if}\ r>0.
\end{cases}
\]
Both products can be extended linearly to $V.$

Let $O_{g}\left(V\right)$ be the linear span of all $u\circ_{g}v$
and define the linear space
\[
A_{g}\left(V\right)=V/O_{g}\left(V\right).
\]
For short we will denote the image of $v$ in $A_{g}(V)$ by $[v].$

In particular, when $g=1$, $A_{g}\left(V\right),$ $O_{g}\left(V\right)$
are just $A\left(V\right)$, $O\left(V\right)$ in \cite{Z}.

The following theorms are given in \cite{DLM3}.

\begin{theorem} \label{t1} (1) $If$ $r\not=0$, then $V^{r}\subseteq O_{g}\left(V\right).$
That is, $A_{g}\left(V\right)$ is a quotient of $A(V^{0}).$

(2) The product $\ast_{g}$ induces the structure of an associative
algebra on $A_{g}\left(V\right)$ with identity $[{\bf 1}]$ and central
element $[\omega].$

(3) There is a bijective correspondence between the sets of equivalence
classes of simple $A_{g}\left(V\right)$-modules and the sets of equivalence
classes of simple admissible $g$-twisted $V$-modules.

(4) If $V$ is $g$-rational then there are only finitely many inequivalent
irreducible $g$-twisted admissible $V$-modules and every irreducible
$g$-twisted admissible $V$-module is ordinary.

(5) If $V$ is $g$-rational then $A_{g}\left(V\right)$ is a finite
dimensional semisimple associative algebra. \end{theorem}

\section{$A_{g}\left(V^{\otimes k}\right)$ and $A\left(V\right)$}

Let $g=\left(1\ 2\ \cdots k\right)$ be a $k$-cycle which is naturally
an automorphism of the tensor product vertex operator algebra $V^{\otimes k}$
where $k$ is a fixed positive integer. In this section, we first
discuss some properties of $\Delta_{k}\left(z\right)$ defined in
\cite{BDM}. Then we prove that $A_{g}\left(V^{\otimes k}\right)$
is isomorphic to $A\left(V\right)$.

\subsection{$\Delta$-operators}

Recall from \cite{BDM} the operator $\Delta_{k}\left(z\right)$ on
$V$: Let $\mathbb{Z}_{+}$ be the set of positive integers and $x,z$
be formal variables commuting with each other. In $\left(\text{End}V\right)\left[\left[z^{1/k},z^{-1/k}\right]\right],$
set
\[
\Delta_{k}\left(z\right)=e^{\sum_{j\in\mathbb{Z}_{+}}a_{j}z^{-j/k}L\left(j\right)}k^{-L\left(0\right)}z^{\left(1/k-1\right)L\left(0\right)}
\]
where $a_{j}\in\mathbb{C}$, $j\in\mathbb{Z}_{+}$ satisfies
\[
e^{-\sum_{j\in\mathbb{Z}_{+}}a_{j}x^{j+1}\frac{\partial}{\partial x}}\cdot x=\frac{1}{k}\left(1+x\right)^{k}-\frac{1}{k}.
\]

The following property of $\Delta_{k}\left(z\right)$ from \cite{BDM}
will be useful in this paper.

\begin{proposition}\label{ Proposition 2.2 of =00003D00003D00005BBDM=00003D00003D00005D}
In $\left(\text{End}V\right)\left[\left[z^{1/k},z^{-1/k}\right]\right]$,
for all $u\in V$ we have

\[
\Delta_{k}\left(z\right)Y\left(u,x\right)\Delta_{k}\left(z\right)^{-1}=Y\left(\Delta_{k}\left(z+x\right)u,\left(z+x\right)^{1/k}-z^{1/k}\right).
\]
\end{proposition}

In particular, $\Delta_{k}\left(1\right)=e^{\sum_{j\in\mathbb{Z}_{+}}a_{j}L\left(j\right)}k^{-L\left(0\right)}$
is a well defined operator on $V$ and for homogeneous $u\in V$,
we have
\[
\Delta_{k}\left(1\right)u=e{}^{\sum_{j\in\mathbb{Z}_{+}}a_{j}L\left(j\right)}k^{-L\left(0\right)}u=e^{\sum_{j\in\mathbb{Z}_{+}}a_{j}L\left(j\right)}uk^{-\text{wt}u}.
\]

Note that $\Delta_{k}\left(1\right)$ is an invertible operator with
\begin{equation}
\Delta_{k}\left(1\right)^{-1}=e^{\sum_{j\in\mathbb{Z}_{+}}-a_{j}L\left(j\right)}uk^{\text{wt}u}.\label{delta operator inverse}
\end{equation}

\subsection{The isomorphism between $A_{g}\left(V^{\otimes k}\right)$ and $A$$\left(V\right)$}

In this section, we construct an isomorphism between $A_{g}\left(V^{\otimes k}\right)$
and $A\left(V\right)$ as associative algebras.

Let $\eta=e^{\frac{2\pi i}{k}}$ be the $k^{\text{th}}$ root of unity.

\begin{lemma} \label{lemma for linear systerm} Let $W$ be a vector
space and $x_{i},u_{s}\in W$ with $i,s=1,2,\cdots,k-1$. Then the
linear system
\[
\sum_{i=1}^{k-1}\eta^{is}x_{i}=u_{s}\ \left(s=1,\cdots,k-1\right)
\]
has a unique solution
\begin{equation}
x_{i}=\sum_{t=1}^{k-1}\frac{1}{k}\left(\eta^{-it}-1\right)u_{t\ }.\label{solution(*)}
\end{equation}

\end{lemma}
\begin{proof}
It is clear that the linear system has a unique solution. It is easy
to verify that
\begin{alignat*}{1}
 & \sum_{j=1}^{k-1}\eta^{js}x_{j}\\
= & \sum_{j=1}^{k-1}\eta^{js}\left(\frac{1}{k}\sum_{t=1}^{k-1}\eta^{-jt}u_{t}-\frac{1}{k}\sum_{t=1}^{k-1}u_{t}\right)\\
= & \sum_{j=1}^{k-1}\eta^{js}\frac{1}{k}\sum_{t=1}^{k-1}\eta^{-jt}u_{t}-\left(\sum_{j=1}^{k-1}\eta^{js}\right)\left(\frac{1}{k}\sum_{t=1}^{k-1}u_{t}\right)\\
= & \sum_{j=1}^{k-1}\eta^{js}\frac{1}{k}\sum_{t=1}^{k-1}\eta^{-jt}u_{t}+\frac{1}{k}\sum_{t=1}^{k-1}u_{t}\\
= & \frac{1}{k}\sum_{t=1}^{k-1}\left(\sum_{j=1}^{k-1}\eta^{j\left(s-t\right)}+1\right)u_{t}\\
= & u_{s}
\end{alignat*}
where we use $\sum_{j=1}^{k-1}\eta^{js}=-1$ and $\sum_{j=1}^{k-1}\eta^{j\left(s-t\right)}=\begin{cases}
k-1 & s=t\\
-1 & s\not=t
\end{cases}$. Therefore (\ref{solution(*)}) satisfies the linear system.
\end{proof}
For any $n\in\left\{ 1,2,\cdots,k\right\} $. Let $v_{1},\cdots,v_{n}\in V$,
$a_{1},\cdots,a_{n}\in\left\{ 1,2,\cdots k\right\} $. We denote $x_{v_{1},\cdots,v_{n}}^{a_{1},\cdots,a_{n}}$
the vector whose $a_{j}^{\text{th}}$-tensor factor is $v_{j}$ and
whose other tensor factors are ${\bf 1}.$ We call $x_{v_{1},\cdots,v_{n}}^{a_{1},\cdots,a_{n}}$
a $n$-tensor vector. In particular, we will denote $x_{u}^{a}$ by
$u^{a}$ which is a $1$-tensor vector whose $a^{\text{th }}$ tensor
factor is $u$ and whose other tensor factors are ${\bf 1}$.

\begin{lemma} \label{n-tensor to n-1-tensor} Let $g=\left(1\ 2\cdots k\right).$
$A_{g}\left(V^{\otimes k}\right)$ is spanned by $[\sum_{a=1}^{k}u^{a}]$
for $u\in V$.

\end{lemma}
\begin{proof}
First by Theorem \ref{t1}, $A_{g}\left(V^{\otimes k}\right)$ is
spanned by vectors of form $\sum_{a=1}^{k}g^{a}x_{v_{1},\cdots,v_{n}}^{a_{1},\cdots,a_{n}}$
modulo $O_{g}(V)$ for $v_{i}\in V.$ Let $n=2,\cdots,k$, we will
prove that any $n$-tensor vector of the form $\sum_{i=1}^{k}g^{a}\left(x_{v_{1},\cdots,v_{n}}^{a_{1},\cdots,a_{n}}\right)$
can be reduced to a $\left(n-1\right)$-tensor vector modulo $O_{g}(V)$
and hence by induction $A_{g}\left(V^{\otimes k}\right)$ is spanned
by $1$-tensor vectors of the form $\sum_{a=1}^{k}u^{a},$ for any
$u\in V$. Here we only give a proof for $(a_{1},\cdots,a_{n})=(1,\cdots,n)$
and the proof for general case is similar.

Note that $\sum_{a=1}^{k}\eta^{-\left(a-1\right)s}u^{a}\in\left(V^{\otimes k}\right)^{s}$
and $\sum_{b=1}^{k}\omega^{\left(b-1\right)s}g^{b-1}x_{v_{1},\cdots,v_{n-1}}^{1,\cdots,n-1}\in\left(V^{\otimes k}\right)^{k-s}$
with $s=1,...,k-1.$ Using
\[
g^{b-1}x_{v_{1},\cdots,v_{n-1}}^{1,\cdots,n-1}=x_{v_{1},\cdots,v_{n-1}}^{b,\cdots,b+n-2}
\]
where $b+m$ is understood to be $b+m-k$ if $b+m>k$ we have
\begin{alignat*}{1}
0\equiv & \left(\sum_{a=1}^{k}\eta^{-\left(a-1\right)s}u^{a}\right)\text{\ensuremath{\circ}}_{g}\left(\sum_{b=1}^{k}\eta^{\left(b-1\right)s}g^{b-1}x_{v_{1},\cdots,v_{n-1}}^{1,\cdots,n-1}\right)\\
= & \Res_{z}Y\left(\sum_{a=1}^{k}\eta^{-\left(a-1\right)s}u^{a},z\right)\left(\sum_{b=1}^{k}\eta^{\left(b-1\right)s}x_{v_{1},\cdots,v_{n-1}}^{b,\cdots,b+n-2}\right)\frac{\left(1+z\right)^{\text{wt}u-1+\frac{s}{k}}}{z}\\
= & \sum_{a=1}^{k}\sum_{b=1}^{k}\eta^{\left(b-a\right)s}\Res_{z}\left(Y\left(u^{a},z\right)x_{v_{1},\cdots,v_{n-1}}^{b,\cdots,b+n-2}\right)\frac{\left(1+z\right)^{\text{wt}u-1+\frac{s}{k}}}{z}\\
= & \sum_{b=1}^{k}\sum_{a\notin\left\{ b,b+1,\cdots,b+n-2\right\} }\eta^{\left(b-a\right)s}x_{u,v_{1},\cdots,v_{n-1}}^{a,b,b+1,\cdots,b+n-2}\\
 & +\sum_{b=1}^{k}\sum_{i=1}^{n-1}\eta^{-\left(i-1\right)s}x_{v_{1},\cdots,v_{i-1},v_{i}',v_{i+1},\cdots,v_{n-1}}^{b,\cdots,b+n-2}
\end{alignat*}
where $v_{i}'=\Res_{Y}\left(u,z\right)v_{i}\frac{\left(1+z\right)^{\text{wt}u-1+\frac{s}{k}}}{z}.$
Set
\[
y_{i}=\sum_{b=1}^{k}g^{b}x_{u,v_{1},\cdots,v_{n-1}}^{n-1+i,1,\cdots,n-1},
\]
\[
u_{s}=-\sum_{b=1}^{k}\sum_{i=1}^{n-1}\eta^{-\left(i-1\right)s}x_{v_{1},\cdots,v_{i-1},v_{i}',v_{i+1},\cdots,v_{n-1}}^{b,\cdots,b+n-2}
\]
for $i=1,...,k-n+1.$ Then
\[
\sum_{b=1}^{k}\sum_{a\notin\left\{ b,b+1,\cdots,b+n-2\right\} }\eta^{\left(b-a\right)s}x_{u,v_{1},\cdots,v_{n-1}}^{a,b,b+1,\cdots,b+n-2}=\sum_{i=1}^{k-n+1}\eta^{-(n+i-2)s}y_{i}\equiv u_{s}.
\]
Clearly, this linear system has a solution that each $y_{i}$ is a
linear combination of the $u_{s}$ for all $i.$

By induction on $n$ we see that any vector in $A_{g}$$\left(V^{\otimes k}\right)$
is spanned by vectors of the form $[\sum_{a=1}^{k}u^{a}]$ for $u\in V.$
The proof is complete.
\end{proof}
In the proof of Lemma \ref{n-tensor to n-1-tensor} we did not give
an explicit expression of $y_{i}$ in terms of $u_{s}.$ But for the
later purpose we need an explicit expression of $y_{i}$ for $n=2$
and $i=1,...k-1.$ To use Lemma \ref{lemma for linear systerm} ,
we set $x_{i}=y_{k-i}$ in this case. \begin{lemma}\label{mainlem}
For $i,s=1,...,k-1$ we have
\[
x_{i}=\sum_{m=0}^{k-1}g^{m}x_{u,v}^{1,1+i}=\sum_{m=1}^{k}x_{u,v}^{m,m+i}\equiv\sum_{s=1}^{k-1}\frac{1}{k}\left(\omega^{-is}-1\right)u_{s}
\]
\[
u_{s}=-\sum_{j=1}^{k}\Res_{z}\left(Y\left(u,z\right)v\frac{\left(1+z\right)^{\text{wt}u-1+\frac{s}{k}}}{z}\right)^{j}.
\]
\end{lemma}
\begin{proof}
From the proof of Lemma \ref{n-tensor to n-1-tensor} we see that
for $n=2,$
\[
\sum_{i=1}^{k-1}\eta^{-is}y_{i}=\sum_{i=1}^{k-1}\eta^{is}x_{i}\equiv u_{s}.
\]
Now applying Lemma \ref{lemma for linear systerm} gives the desired
result.
\end{proof}
\begin{remark} By Lemma \ref{n-tensor to n-1-tensor}, to construct
isomorphism from $A_{g}\left(V^{\otimes k}\right)$ to $A\left(V\right)$,
it suffices to determine the images of $[\sum_{a=1}^{k}u^{a}]$ for
$u\in V.$ \end{remark}

\begin{lemma} \label{calculation lemma} Let $k$ be a positive integer.
Then
\[
\left(\frac{1}{\left(\left(1+z\right)^{k}-1\right)^{2}}+\sum_{t=1}^{k-1}\frac{\frac{t}{k}\left(1+z\right)^{-k+t}+\left(1-\frac{t}{k}\right)\left(1+z\right)^{t}}{\left(\left(1+z\right)^{k}-1\right)^{2}}\right)k\left(1+z\right)^{k-1}=\frac{1}{z^{2}}
\]
\end{lemma}
\begin{proof}
It suffices to prove that
\[
kz^{2}\left(1+z\right)^{k-1}\left(1+\sum_{t=1}^{k-1}\frac{t}{k}\left(1+z\right)^{-k+t}+\sum_{t=1}^{k-1}\left(1-\frac{t}{k}\right)\left(1+z\right)^{t}\right)=\left(\left(1+z\right)^{k}-1\right)^{2}.
\]
Noting that $z^{2}=\left(1+z\right)^{2}-2\left(1+z\right)+1,$ we
get
\begin{alignat*}{1}
 & z^{2}\left(1+z\right)^{k-1}\left(1+\sum_{t=1}^{k-1}\frac{t}{k}\left(1+z\right)^{-k+t}+\sum_{t=1}^{k-1}\left(1-\frac{t}{k}\right)\left(1+z\right)^{t}\right)\\
 & =\left[\left(1+z\right)^{k+1}-2\left(1+z\right)^{k}+\left(1+z\right)^{k-1}\right]\\
 & \cdot\left(1+\sum_{t=1}^{k-1}\frac{t}{k}\left(1+z\right)^{-k+t}+\sum_{t=1}^{k-1}\left(1-\frac{t}{k}\right)\left(1+z\right)^{t}\right)\\
 & =\left(1+z\right)^{k+1}-2\left(1+z\right)^{k}+\left(1+z\right)^{k-1}\\
 & +\sum_{t=1}^{k-1}\frac{t}{k}\left(1+z\right)^{1+t}-2\sum_{t=1}^{k-1}\frac{t}{k}\left(1+z\right)^{t}+\sum_{t=1}^{k-1}\frac{t}{k}\left(1+z\right)^{t-1}\\
 & +\sum_{t=1}^{k-1}\left(1-\frac{t}{k}\right)\left(1+z\right)^{k+t+1}-2\sum_{t=1}^{k-1}\left(1-\frac{t}{k}\right)\left(1+z\right)^{k+t}+\sum_{t=1}^{k-1}\left(1-\frac{t}{k}\right)\left(1+z\right)^{k+t-1}\\
= & \frac{1}{k}\left[\left(1+z\right)^{2k}-2\left(1+z\right)^{k}+1\right]\\
= & \frac{1}{k}\left(\left(1+z\right)^{k}-1\right)^{2},
\end{alignat*}
as desired.
\end{proof}
Now we prove the main theorem.

\begin{theorem}\label{mainthm} Define
\begin{alignat*}{1}
\phi: & A_{g}\left(V^{\otimes k}\right)\to A\left(V\right)\\
 & [\sum_{a=1}^{k}u^{a}]\mapsto[k\Delta_{k}\left(1\right)u]
\end{alignat*}
Then $\phi$ gives an isomorphism between $A_{g}\left(V^{\otimes k}\right)$
and $A\left(V\right).$ \end{theorem}
\begin{proof}
Recall that $A_{g}\left(V^{\otimes k}\right)=V^{\otimes k}/O_{g}\left(V^{\otimes k}\right)$
and $A\left(V\right)=V/O\left(V\right)$ where $O_{g}\left(V^{\otimes k}\right)$
and $O\left(V\right)$ are defined in Section \ref{subsec:The-associative-algebra}.
To show that $\phi$ is an isomorphism between the associative algebras
$A_{g}\left(V^{\otimes k}\right)$ and $A\left(V\right)$, we first
need to show that $\phi$ is well-defined.

Since $A_{g}(V^{\otimes k})=A((V^{\otimes k})^{0})/O_{g}(V)\cap(V^{\otimes k})^{0})$
and $O_{g}(V)\cap(V^{\otimes k})^{0})$ is spanned by $u\circ_{g}v$
for $u\in(V^{\otimes k})^{s}$ and $v\in(V^{\otimes k})^{k-s},$ we
simply map $u\circ_{g}v$ to $0$ for $r=1,...,k-1.$ In fact, from
the proof of Lemma \ref{n-tensor to n-1-tensor} we see that $[u\circ_{g}v]=0$
just gives an identification between $p$-tensor vectors and $q$-tensor
vectors in $A_{g}(V^{\otimes k})$ such that either $p>1$ or $q>1.$
So the main task is to show that for any $u,v\in V^{\otimes k}$,
$\phi\left(u\circ_{g}v\right)$$\in O\left(V\right).$

Let $\overline{u}=\sum_{a=1}^{k}u^{a}$, $\overline{v}=\sum_{b=1}^{k}v^{b}\in\left(V^{\otimes k}\right)^{0}.$
Then
\[
\overline{u}\circ\overline{v}=\sum_{a,b=1}^{k}Y\left(u^{a},z\right)v^{b}\frac{\left(1+z\right)^{\text{wt}u}}{z^{2}}=\sum_{i=1}^{k}\left(u\circ v\right)^{i}+\text{wt}u\sum_{i=1}^{k-1}x_{i}+\sum_{i=1}^{k-1}y_{i}
\]
where
\[
x_{i}=\sum_{m=0}^{k-1}g^{m}x_{u,v}^{1,i}=\sum_{m=1}^{k}x_{u,v}^{m,m+i}\equiv\sum_{t=1}^{k-1}\frac{1}{k}\left(\omega^{-it}-1\right)u_{t},
\]
\[
\ y_{i}=\sum_{m=0}^{k-1}g^{m}x_{u_{-2}1,v}^{1,i}=\sum_{m=1}^{k}x_{u_{-2}1,v}^{m,m+i}\equiv\sum_{t=1}^{k-1}\frac{1}{k}\left(\omega^{-it}-1\right)w_{t}
\]
with
\[
u_{t}=-\sum_{j=1}^{k}\left(Y\left(u,z\right)v\frac{\left(1+z\right)^{\text{wt}u-1+\frac{t}{k}}}{z}\right)^{j},
\]
and
\[
w_{t}=-\sum_{j=1}^{k}\left(Y\left(u_{-2}1,z\right)v\frac{\left(1+z\right)^{\text{wt}u+\frac{t}{k}}}{z}\right)^{j}
\]
by Lemma \ref{mainlem}. Then
\begin{align*}
 & \phi\left(\overline{u}\circ\overline{v}\right)\\
= & k\Delta_{k}\left(1\right)u\circ v-\text{wt}u\sum_{t,i=1}^{k-1}\frac{1}{k}\left(\omega^{-it}-1\right)k\Delta_{k}\left(1\right)\text{Res}_{z}Y\left(u,z\right)v\frac{\left(1+z\right)^{\text{wt}u-1+\frac{t}{k}}}{z}\\
 & -\sum_{t,i=1}^{k-1}\frac{1}{k}\left(\omega^{-it}-1\right)k\Delta_{k}\left(1\right)\text{Res}_{z}Y\left(u_{-2}1,z\right)v\frac{\left(1+z\right)^{\text{wt}u+\frac{t}{k}}}{z}\\
= & k\Delta_{k}\left(1\right)u\circ v+\text{wt}u\sum_{t}^{k-1}k\Delta_{k}\left(1\right)\text{Res}_{z}Y\left(u,z\right)v\frac{\left(1+z\right)^{\text{wt}u-1+\frac{t}{k}}}{z}\\
 & -\sum_{t=1}^{k-1}k\Delta_{k}\left(1\right)\text{Res}_{z}Y\left(u,z\right)v\left(\left(\text{wt}u+\frac{t}{k}\right)\frac{\left(1+z\right)^{\text{wt}u-1+\frac{t}{k}}}{z}-\frac{\left(1+z\right)^{\text{wt}u+\frac{t}{k}}}{z^{2}}\right)\\
= & k\Delta_{k}\left(1\right)u\circ v-\sum_{t=1}^{k-1}k\Delta_{k}\left(1\right)\text{Res}_{z}Y\left(u,z\right)v\left[\frac{t}{k}\frac{\left(1+z\right)^{\text{wt}u-1+\frac{t}{k}}}{z}-\frac{\left(1+z\right)^{\text{wt}u+\frac{t}{k}}}{z^{2}}\right]\\
= & k\Delta_{k}\left(1\right)\text{Res}_{z}Y\left(u,z\right)v\left(1+z\right)^{\text{wt}u}\left(\frac{1}{z^{2}}-\sum_{t=1}^{k-1}\frac{t}{k}\frac{\left(1+z\right)^{-1+\frac{t}{k}}}{z}+\sum_{t=1}^{k-1}\frac{\left(1+z\right)^{\frac{t}{k}}}{z^{2}}\right)\\
= & \text{Res}_{z}k\Delta_{k}\left(1\right)Y\left(u,z\right)\Delta_{k}^{-1}\left(1\right)\Delta_{k}\left(1\right)v\left(1+z\right)^{\text{wt}u}\left(\frac{1}{z^{2}}-\sum_{t=1}^{k-1}\frac{t}{k}\frac{\left(1+z\right)^{-1+\frac{t}{k}}}{z}+\sum_{t=1}^{k-1}\frac{\left(1+z\right)^{\frac{t}{k}}}{z^{2}}\right)\\
= & \text{Res}_{z}kY\left(\Delta_{k}\left(1+z\right)u,\left(1+z\right)^{\frac{1}{k}}-1\right)\Delta_{k}\left(1\right)v\left(1+z\right)^{\text{wt}u}\\
 & \left(\frac{1}{z^{2}}-\sum_{t=1}^{k-1}\frac{t}{k}\frac{\left(1+z\right)^{-1+\frac{t}{k}}}{z}+\sum_{t=1}^{k-1}\frac{\left(1+z\right)^{\frac{t}{k}}}{z^{2}}\right)\\
= & k^{1-\text{wt}u}\text{Res}_{z}Y\left(e^{\sum_{j\in\mathbb{Z}_{+}}a_{j}(1+z)^{-j/k}L\left(j\right)}u,\left(1+z\right)^{\frac{1}{k}}-1\right)\Delta_{k}\left(1\right)v\left(1+z\right)^{\frac{1}{k}\text{wt}u}\\
 & \left(\frac{1}{z^{2}}-\sum_{t=1}^{k-1}\frac{t}{k}\frac{\left(1+z\right)^{-1+\frac{t}{k}}}{z}+\sum_{t=1}^{k-1}\frac{\left(1+z\right)^{\frac{t}{k}}}{z^{2}}\right)\\
= & k^{1-\text{wt}u}\text{Res}_{z}Y\left(e^{\sum_{j\in\mathbb{Z}_{+}}a_{j}(1+z)^{-j}L\left(j\right)}u,z\right)\Delta_{k}\left(1\right)v\left(1+z\right)^{\text{wt}u}\\
 & \left(\frac{1}{\left(\left(1+z\right)^{k}-1\right)^{2}}-\sum_{t=1}^{k-1}\frac{t}{k}\frac{\left(1+z\right)^{-k+t}}{\left(1+z\right)^{k}-1}+\sum_{t=1}^{k-1}\frac{\left(1+z\right)^{t}}{\left(\left(1+z\right)^{k}-1\right)^{2}}\right)k\left(1+z\right)^{k-1}\\
= & k^{1-\text{wt}u}\text{Res}_{z}Y\left(e^{\sum_{j\in\mathbb{Z}_{+}}a_{j}(1+z)^{-j}L\left(j\right)}u,z\right)\Delta_{k}\left(1\right)v\left(1+z\right)^{\text{wt}u}\\
 & \left(\frac{1}{\left(\left(1+z\right)^{k}-1\right)^{2}}+\sum_{t=1}^{k-1}\frac{\frac{t}{k}\left(1+z\right)^{-k+t}+\left(1-\frac{t}{k}\right)\left(1+z\right)^{t}}{\left(\left(1+z\right)^{k}-1\right)^{2}}\right)k\left(1+z\right)^{k-1}\\
= & k^{1-\text{wt}u}\text{Res}_{z}Y\left(e^{\sum_{j\in\mathbb{Z}_{+}}a_{j}(1+z)^{-j}L\left(j\right)}u,z\right)\Delta_{k}\left(1\right)v\left(1+z\right)^{\text{wt}u}\frac{1}{z^{2}}\in O\left(V\right)
\end{align*}
where we apply Proposition \ref{ Proposition 2.2 of =00003D00003D00005BBDM=00003D00003D00005D}
in the sixth identity, Lemma \ref{change of variable formula} with
substitution $x=\left(1+z\right)^{\frac{1}{k}}-1$ in the eighth identity,
and in the last identity we use Lemma \ref{calculation lemma}. Thus
we obtain $\phi\left(\overline{u}\circ\overline{v}\right)=\lambda u\circ v\in O\left(V\right)$
for some constant $\lambda.$

Now we prove that $\phi$ is a homomorphism. That is, $\phi\left(\overline{u}\ast\overline{v}\right)=\phi\left(\overline{u}\right)\ast\phi\left(\overline{v}\right)=\left(k\Delta_{k}\left(1\right)u\right)\ast$$\left(k\Delta_{k}\left(1\right)v\right).$
We first have
\begin{align*}
\overline{u}\ast\overline{v} & =Y\left(\sum_{a=1}^{k}u^{a},z\right)\left(\sum_{b=1}^{k}v^{b}\right)\frac{\left(1+z\right)^{\text{wt}u}}{z}\\
 & =\sum_{a\not=b}x_{u,v}^{a,b}+\sum_{i=1}^{k}\left(u\ast v\right)^{i}\\
 & =\sum_{b-a=i,i=1,\cdots,k-1}x_{u,v}^{a,b}+\sum_{i=1}^{k}\left(u\ast v\right)^{i}\quad\\
 & =\sum_{a=1}^{k}\sum_{i=1}^{k-1}x_{u,v}^{a,a+i}+\sum_{i=1}^{k}\left(u\ast v\right)^{i}\\
 & =\sum_{i=1}^{k-1}x_{i}+\sum_{i=1}^{k}\left(u\ast v\right)^{i}
\end{align*}
where $x_{i}=\sum_{a=1}^{k}x_{u,v}^{a,a+i}$.

As before we have
\begin{align*}
 & \phi\left(\overline{u}\ast\overline{v}\right)\\
= & k\Delta_{k}\left(1\right)\left(u\ast v\right)-\frac{1}{k}\sum_{t,i=1}^{k-1}\left(\omega^{-it}-1\right)k\Delta_{k}\left(1\right)\text{Res}_{z}Y\left(u,z\right)v\frac{\left(1+z\right)^{\text{wt}u-1+\frac{t}{k}}}{z}\\
= & k\Delta_{k}\left(1\right)\text{Res}_{z}Y\left(u,z\right)v\left(1+z\right)^{\text{wt}u}\left(\frac{1}{z}+\sum_{t=1}^{k-1}\frac{\left(1+z\right)^{\text{wt}u-1+\frac{t}{k}}}{z}\right)\\
= & k\text{Res}_{z}Y\left(e^{\sum_{j}a_{j}\left(1+z\right)^{-j/k}L\left(j\right)}u,\left(1+z\right)^{1/k}-1\right)\Delta_{k}\left(1\right)v\left(1+z\right)^{\text{wt}u}\\
 & \cdot k^{-\text{wt}u}\left(1+z\right)^{\left(1/k-1\right)\text{wt}u}\left(\frac{1}{z}+\sum_{t=1}^{k-1}\frac{\left(1+z\right)^{-t/k}}{z}\right)\\
= & k^{1-\text{wt}u}\text{Res}_{z}Y\left(e^{\sum a_{j}\left(1+z\right)^{-j}L\left(j\right)}u,z\right)\Delta_{k}\left(1\right)v\left(1+z\right)^{\text{wt}v}\\
 & \cdot\left(\frac{1}{\left(1+z\right)^{k}-1}+\sum_{t=1}^{k-1}\frac{\left(1+z\right)^{-t}}{\left(1+z\right)^{k}-1}\right)k\left(1+z\right)^{k-1}\\
= & k^{1-\text{wt}u}\text{Res}_{z}Y\left(e^{\sum a_{j}\left(1+z\right)^{-j}L\left(j\right)}u,z\right)\Delta_{k}\left(1\right)v\left(1+z\right)^{\text{wt}v}\\
 & \cdot\left(\frac{\left(1+z\right)^{k-1}}{\left(1+z\right)^{k}-1}+\sum_{t=1}^{k-1}\frac{\left(1+z\right)^{k-1-t}}{\left(1+z\right)^{k}-1}\right)k\\
= & k^{2-\text{wt}u}\text{Res}_{z}Y\left(e^{\sum a_{j}\left(1+z\right)^{-j}L\left(j\right)}u,z\right)\Delta_{k}\left(1\right)v\left(1+z\right)^{\text{wt}v}\cdot\frac{\sum_{t=0}^{k-1}\left(1+z\right)^{t}}{\left(1+z\right)^{k}-1}\\
= & k^{2-\text{wt}u}\text{Res}_{z}Y\left(e^{\sum a_{j}\left(1+z\right)^{-j}L\left(j\right)}u,z\right)\Delta_{k}\left(1\right)v\left(1+z\right)^{\text{wt}v}\frac{1}{z}\\
= & k^{2}\text{Res}_{z}Y\left(\Delta_{k}\left(1\right)u,z\right)\Delta_{k}\left(1\right)v\left(1+z\right)^{\text{wt}v}\frac{1}{z}\\
= & \left(k\Delta_{k}\left(1\right)u\right)\ast\left(k\Delta_{k}\left(1\right)v\right)\\
= & \phi\left(\overline{u}\right)\ast\phi\left(\overline{v}\right).
\end{align*}

Since that $\Delta_{k}\left(1\right)$ is an invertible operator as
given in (\ref{delta operator inverse}), we define
\begin{alignat*}{1}
\psi: & A\left(V\right)\to A_{g}\left(V^{\otimes k}\right)\\
 & u\mapsto\frac{1}{k}\sum_{a=1}^{k}\left(\Delta_{k}\left(1\right)^{-1}u\right)^{a}.
\end{alignat*}
Thus $\psi$ is the inverse of $\phi$ and $\phi$ defines an isomorphism.
\end{proof}
Now we deal with an arbitrary element $g\in S_{k},$ Then $g=g_{1}\cdots g_{s}$
is a product of disjoint cycles. If $U,W$ are vertex operator algebras
with automorphisms $f,h$ respectively, then $f\otimes h$ is an automorphism
of $U\otimes W$ and $A_{f\otimes h}(U\otimes W)$ is isomorphic to
$A_{f}(U)\otimes_{\C}A_{h}(W).$ Using this fact and Theorem \ref{mainthm}
we conclude this paper with the following corollary. \begin{corollary}
Let $g=g_{1}\cdots g_{s}$ be a product of disjoint cycles. Then $A_{g}(V^{\otimes k})$
is isomorphic to $A(V)^{\otimes s}$ where is the tensor is over $\C.$
\end{corollary}

{\footnotesize{}{}{}C. Dong: Department of Mathematics, University
of California Santa Cruz, CA 95064 USA; }\emph{\footnotesize{}{}{}dong@ucsc.edu}{\footnotesize\par}

{\footnotesize{}{}{}F. Xu: University of California at Riverside,
Riverside, CA 92521 USA; }\emph{\footnotesize{}{}{}xufeng@math.ucr.edu}{\footnotesize\par}

{\footnotesize{}{}{}N. Yu: School of Mathematical Sciences, Xiamen
University, Fujian, 361005, China; }\emph{\footnotesize{}{}{}ninayu@xmu.edu.cn}{\footnotesize\par}
\end{document}